\UseRawInputEncoding
\documentclass[12pt,reqno]{amsart}
\usepackage{fullpage}

\newtheorem{theorem}{Theorem}[section]

\newtheorem{prop}[theorem]{Proposition}

\newtheorem{conj}[theorem]{Conjecture}

\usepackage{graphicx}
\usepackage{color}
\usepackage[dvipsnames]{xcolor}
\usepackage{subfigure}
\usepackage{amssymb}
\usepackage{amsmath,mathrsfs}
\usepackage{colonequals}
\usepackage{hyperref}
\usepackage{wrapfig}

\theoremstyle{definition}
\newtheorem{definition}[theorem]{Definition}

\newtheorem{example}[theorem]{Example}

\newtheorem{assumption}{Assumption}
\newtheorem{remark}[theorem]{Remark}
\newtheorem{question}[theorem]{Question}
\newtheorem{vmethod}{Voting Method}

\renewcommand{\subset}{\subseteq}

\renewcommand{\epsilon}{\varepsilon}
\renewcommand{\nu}{v}

\newcommand{\abs}[1]{\left|#1\right|}                   
\newcommand{\vnorm}[1]{\left\|#1\right\|}    

\newcommand{\E}{\mathbb{E}}

\renewcommand{\P}{\mathbb{P}}
\newcommand{\R}{\mathbb{R}}

\newcommand{\embolden}[1]{\textbf {#1}}

\renewcommand{\colonequals}{=}

\begin{document}

\title{Designing Stable Elections}

\author{Steven Heilman}
\address{Department of Mathematics, University of Southern California, Los Angeles, CA 90089-2532}
\email{stevenmheilman@gmail.com}
\date{\today}
\thanks{Supported by NSF Grant CCF 1911216 and DMS 1839406.  Any opinions, findings, and conclusions or recommendations expressed in this material are those of the author and do not necessarily reflect the views of the National Science Foundation.}

\maketitle

%
%
%

\section{Introduction}\label{secintro}

Suppose votes have been cast in an election between two candidates, and then an adversary can select a fixed number of votes to change.  Which voting method best preserves the outcome of the election?  A majority vote does, among all voting methods where both candidates have an equal chance of winning the election.

Now, suppose votes have been cast in an election between two candidates, and then each vote is randomly changed with a small probability, independently of the other votes.  It is desirable to keep the outcome of the election the same, regardless of the changes to the votes.  It is well known that the US electoral college system is more than 4 times more likely to have a changed outcome due to vote corruption, when compared to a majority vote.  In fact, Mossel, O'Donnell and Oleszkiewicz proved in 2005 that the majority voting method is most stable to this random vote corruption, among voting methods where each person has a small influence on the election.  Below, we survey the design of elections that are resilient to attempted interference by third parties.  We discuss some recent progress on the analogous result for elections between more than two candidates.  In this case, plurality should be most stable to corruption in votes.  We briefly discuss ranked choice voting methods (where a vote is a ranked list of candidates).

\subsection{Condorcet's Paradox}

Applications of mathematics to the analysis of elections perhaps began with Marquis de Condorcet in the 1700s.  Condorcet's famous paradox demonstrates that an election method that uses ranked preferences of voters might not have a sensible winner.  Consider the following ranking of three candidates $a,b$ and $c$ between three voters $1,2$ and $3$.

\begin{table}[htbp!]
\centering
\begin{tabular}{|c||c|c|c|}
\hline
Voter & Rank 1 & Rank 2 & Rank 3\\
\hline
\hline
1 & $a$ & $b$ & $c$\\
\hline
2 & $b$ & $c$ & $a$\\
\hline
3 & $c$ & $a$ & $b$\\
\hline
\end{tabular}
\caption{Three voters (one for each row of the table) provide rankings of three candidates $a,b$ and $c$.  For example, voter $1$ most prefers candidate $a$.}
\label{table1}
\end{table}

If we ignore candidate $b$, then voters $2$ and $3$ prefer $c$ over $a$, while voter $1$ prefers $a$ over $c$.  So, using a majority rule for these preferences, the voters prefer $c$ over $a$.

\begin{table}[htbp!]
\centering
\begin{tabular}{|c||c|c|}
\hline
Voter & Rank 1 & Rank 2 \\
\hline
\hline
1 & $a$ &  $c$\\
\hline
2 & $c$ & $a$\\
\hline
3 & $c$ & $a$ \\
\hline
\end{tabular}
\caption{If candidate $b$ is ignored in Table \ref{table1}, the remaining rankings of candidates $a$ and $c$ indicate that $c$ is the preferred by the majority of voters.}
\label{table2}
\end{table}

If we ignore candidate $c$ in Table \ref{table1}, then voters $1$ and $3$ prefer $a$ over $b$, while voter $2$ prefers $b$ over $a$.  So, using a majority rule again, the voters prefer $a$ over $b$.

Finally, if we ignore candidate $a$ in Table \ref{table1}, then voters $1$ and $2$ prefer $b$ over $c$, while voter $3$ prefers $c$ over $b$.  So, using a majority rule, the voters prefer $b$ over $c$.

%

\begin{figure}[h!]
   \def\svgwidth{.3\textwidth}
\begingroup%
  \makeatletter%
  \providecommand\color[2][]{%
    \errmessage{(Inkscape) Color is used for the text in Inkscape, but the package 'color.sty' is not loaded}%
    \renewcommand\color[2][]{}%
  }%
  \providecommand\transparent[1]{%
    \errmessage{(Inkscape) Transparency is used (non-zero) for the text in Inkscape, but the package 'transparent.sty' is not loaded}%
    \renewcommand\transparent[1]{}%
  }%
  \providecommand\rotatebox[2]{#2}%
  \newcommand*\fsize{\dimexpr\f@size pt\relax}%
  \newcommand*\lineheight[1]{\fontsize{\fsize}{#1\fsize}\selectfont}%
  \ifx\svgwidth\undefined%
    \setlength{\unitlength}{198.42519685bp}%
    \ifx\svgscale\undefined%
      \relax%
    \else%
      \setlength{\unitlength}{\unitlength * \real{\svgscale}}%
    \fi%
  \else%
    \setlength{\unitlength}{\svgwidth}%
  \fi%
  \global\let\svgwidth\undefined%
  \global\let\svgscale\undefined%
  \makeatother%
  \begin{picture}(1,1)%
    \lineheight{1}%
    \setlength\tabcolsep{0pt}%
    \put(0,0){\includegraphics[width=\unitlength,page=1]{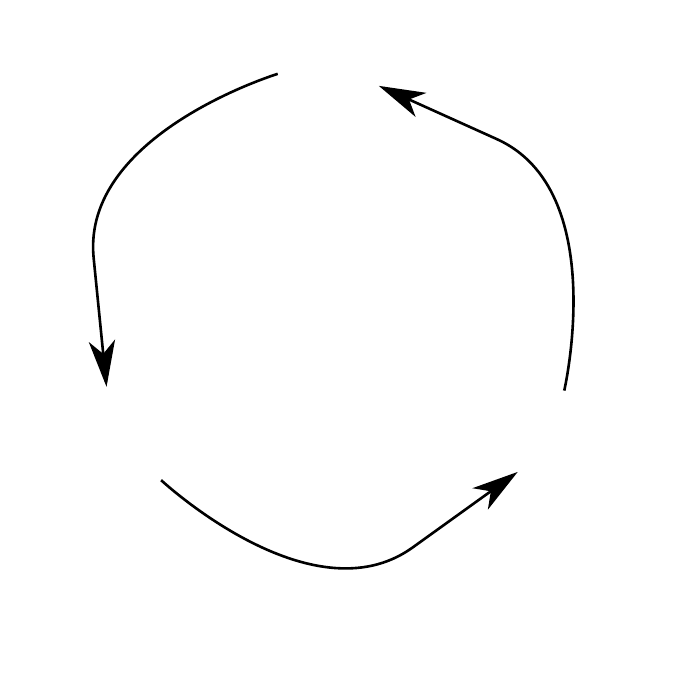}}%
    \put(0.45519134,0.86124596){\color[rgb]{0,0,0}\makebox(0,0)[lt]{\lineheight{1.25}\smash{\begin{tabular}[t]{l}$a$\end{tabular}}}}%
    \put(0.14521182,0.31545303){\color[rgb]{0,0,0}\makebox(0,0)[lt]{\lineheight{1.25}\smash{\begin{tabular}[t]{l}$c$\end{tabular}}}}%
    \put(0.77535213,0.31815275){\color[rgb]{0,0,0}\makebox(0,0)[lt]{\lineheight{1.25}\smash{\begin{tabular}[t]{l}$b$\end{tabular}}}}%
    \put(0.84936652,0.73867277){\color[rgb]{0,0,0}\makebox(0,0)[lt]{\lineheight{1.25}\smash{\begin{tabular}[t]{l}$a$ is preferred \\over $b$\end{tabular}}}}%
    \put(0.46244624,0.05116405){\color[rgb]{0,0,0}\makebox(0,0)[lt]{\lineheight{1.25}\smash{\begin{tabular}[t]{l}$b$ is preferred \\over $c$\end{tabular}}}}%
    \put(0.195163,0.66564551){\color[rgb]{0,0,0}\makebox(0,0)[lt]{\lineheight{1.25}\smash{\begin{tabular}[t]{l}$c$ is pre-\\ferred \\over $a$\end{tabular}}}}%
    \put(-0.21328656,1.37961308){\color[rgb]{0,0,0}\makebox(0,0)[lt]{\begin{minipage}{2.60803564\unitlength}\raggedright \end{minipage}}}%
  \end{picture}%
\endgroup%
 \\

\end{figure}

In conclusion, the voters prefer $a$ over $b$, they prefer $b$ over $c$, and they prefer $c$ over $a$.  So, \textit{no one has won the election}!  This observation is known as Condorcet's paradox.  The simplest way to use rankings of candidates might lead to no one winning the election.

In fact, if we compare pairs of candidates using something other than a majority rule, then some analogue of Condorcet's paradox must still occur, \textit{unless} we ignore all voters except for one (a dictatorship).  This statement can be formalized as Arrow's Impossibility Theorem.

\subsection{Voting Power}

\begin{wrapfigure}{r}{0.5\textwidth}
\vspace{-1cm}
   \def\svgwidth{.5\textwidth}
    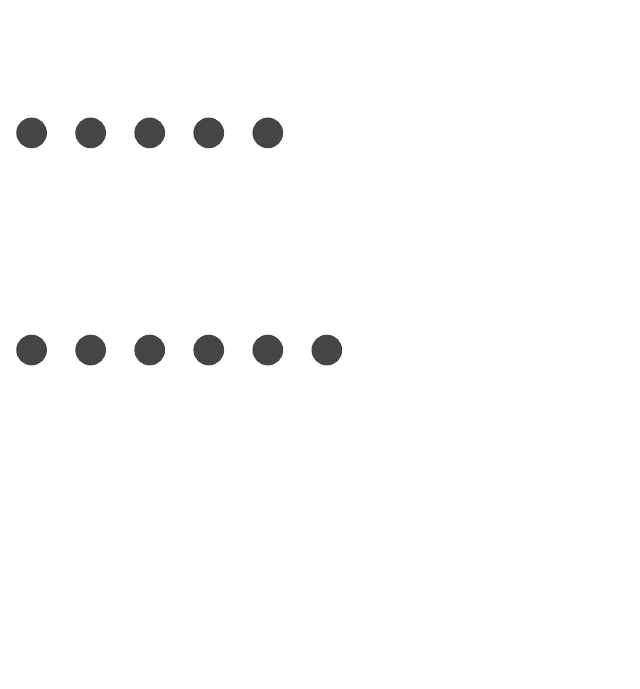 \\
    \vspace{-2cm}
\end{wrapfigure}

Game Theorists such as Shapley, Shubik and Banzhaf in the 1950s and 1960s further developed the mathematical and economical analysis of voting methods.  As an illustrative example, we consider the 1965 restructuring of the UN security council.

\begin{vmethod}[Pre-1965 UN Security Council]
In pre-1965 rules, the UN security council had five permanent members, and six nonpermanent members.  A resolution passes in the security council only if:
\begin{itemize}
\item  All five permanent members want it to pass, and
\item at least \textit{two} nonpermanent members want it to pass.
\end{itemize}
\end{vmethod}

\begin{wrapfigure}{r}{0.5\textwidth}
\vspace{-4cm}
   \def\svgwidth{.5\textwidth}
    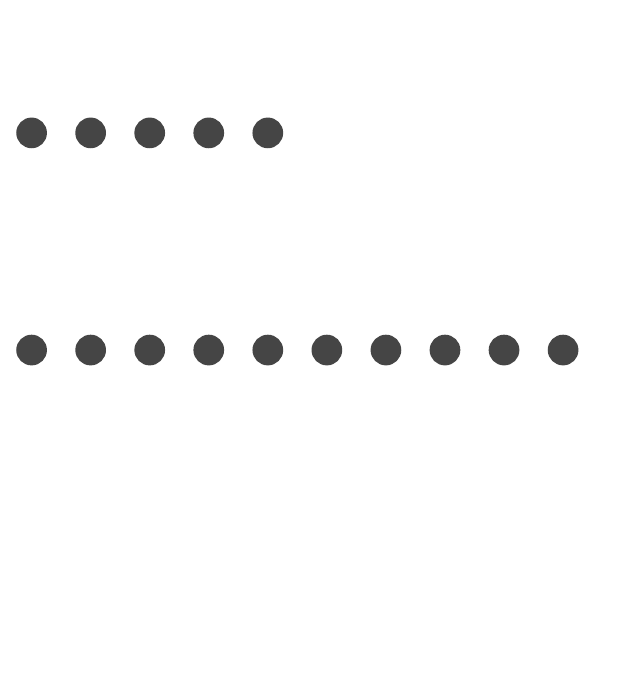 \\
    \vspace{-2cm}
\end{wrapfigure}

In particular, a single permanent member can effectively veto a resolution by voting ``no'' on that resolution.  This voting method was called unfair for the nonpermanent members, so it was restructured in 1965.  After the restructuring, the council had the following form (still in use today).
\begin{vmethod}[Post-1965 UN Security Council]
The UN security council has five permanent members, and now ten nonpermanent members.  A resolution passes in the council only if:
\begin{itemize}
\item  All five permanent members want it to pass, and
\item at least \textit{four} nonpermanent members want it to pass.
\end{itemize}
\end{vmethod}

A rather vague question is then:

\begin{question}\label{q1}
Are the Post-1965 rules more equitable for nonpermanent members of the UN security council than Pre-1965 rules?
\end{question}

There are various ways to answer this question.  One answer, provided by Banzhaf, is to consider the \textbf{power} of a voter in each voting method, i.e. the relative ability of a voter to cause a resolution to pass by changing their vote.  Suppose we label the post-1965 UN security council members by the integers $1$ through $15$, where the numbers $1,2,3,4,5$ represent the five permanent members of the council, and the numbers $6,7,\ldots,15$ represent nonpermament members.  Then, for any integer $i$ between $1$ and $15$, let $b_{i}$ be the number of combinations of votes of members of the council (other than voter $i$), such that when voter $i$ changes their vote from ``no'' to ``yes,'' the resolution changes from not passing to passing.  The \textbf{Banzhaf power index} of a voter $i$ is defined to be the following ratio
$$\frac{b_{i}}{b_{1}+b_{2}+\cdots+b_{15}}.$$

For example, in the post-1965 rules, what would it take for a nonpermanent member to cause the resolution to pass?  First, all permanent members would have to vote ``yes.''  Then, exactly three other nonpermanent members out of nine would vote yes.  So, the number of combinations of votes other members would make is: the number of ways to select $3$ members from a set of $9$, i.e. $\binom{9}{3}=\frac{9\cdot 8\cdot 7}{3\cdot 2}=84$.  So, $b_{6}=b_{7}=\cdots=b_{15}=84$.

In the post-1965 rules, what would it take for a permanent member to cause the resolution to pass?  First, all other permanent members would have to vote ``yes.''  Then, at least four nonpermanent members out of 10 would vote yes.  So, the number of combinations of votes other members would make is: the number of ways to select at least $4$ members from a set of $10$.  This number is $\binom{10}{4}+\binom{10}{5}+\cdots+\binom{10}{10}=848$.  So, $b_{1}=b_{2}=\cdots=b_{5}=848$.

Similar considerations apply for pre-1965 rules.  We summarize the Banzhaf power indices in the following table.


\begin{table}[htbp!]\renewcommand{\arraystretch}{1.5}
\centering
\begin{tabular}{|l|p{5cm}|p{5cm}|c|}
\hline
Voting Method & Banzhaf Power Index for \textit{Non}-Permament Member & Banzhaf Power Index for Permament Member\\
\hline
Pre-1965 Rules & $\frac{5}{6\cdot 5+5\cdot 57}\approx .0159$ & $\frac{57}{6\cdot 5+5\cdot 57}\approx .181$ \\
\hline
Post-1965 Rules & $\frac{84}{10\cdot 84+5\cdot 848}\approx .0165$ & $\frac{848}{10\cdot 84+5\cdot 848}\approx .167$ \\
\hline
\end{tabular}
\caption{Banzhaf Power Indices for UN Security Council Voting Methods}
\end{table}

In summary, the post-1965 rules give more power to non-permanent members, and less power to permanent members of the UN Security Council.  So, according to Banzhaf's definition of voting power, the answer to Question \ref{q1} is: yes.

\subsection{Voting Methods as Functions}

\begin{wrapfigure}{r}{0.5\textwidth}
   \def\svgwidth{.5\textwidth}
    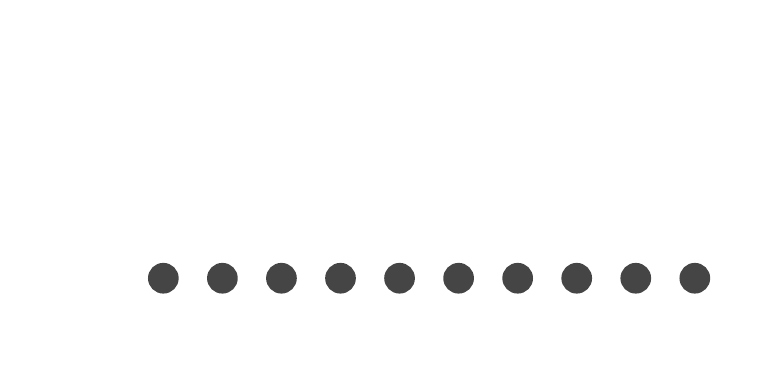 \\
    \vspace{-.5cm}
\end{wrapfigure}

Suppose we run an election between two candidates with $n$ voters, where $n$ is a large integer.  For convenience, we denote the two candidates as $+1$ and $-1$ rather than $a$ and $b$.  If person $i$ votes for candidate $1$, we define $x_{i}=1$, and if person $i$ votes for candidate $-1$, we define $x_{i}=-1$.  We then can then make a list of votes as
$$x=(x_{1},x_{2},\ldots,x_{n}).\qquad\qquad\qquad\qquad\qquad\qquad\qquad\qquad\qquad\qquad$$
A \textbf{voting method} is a function $f$ whose input is the votes $x$ and whose output is the winner of election.  That is, $f(x)=1$ denotes candidate $1$ winning the election when the votes are $x$, and $f(x)=-1$ denotes candidate $-1$ winning the election when the votes are $x$.

\begin{wrapfigure}{r}{0.5\textwidth}
   \def\svgwidth{.5\textwidth}
    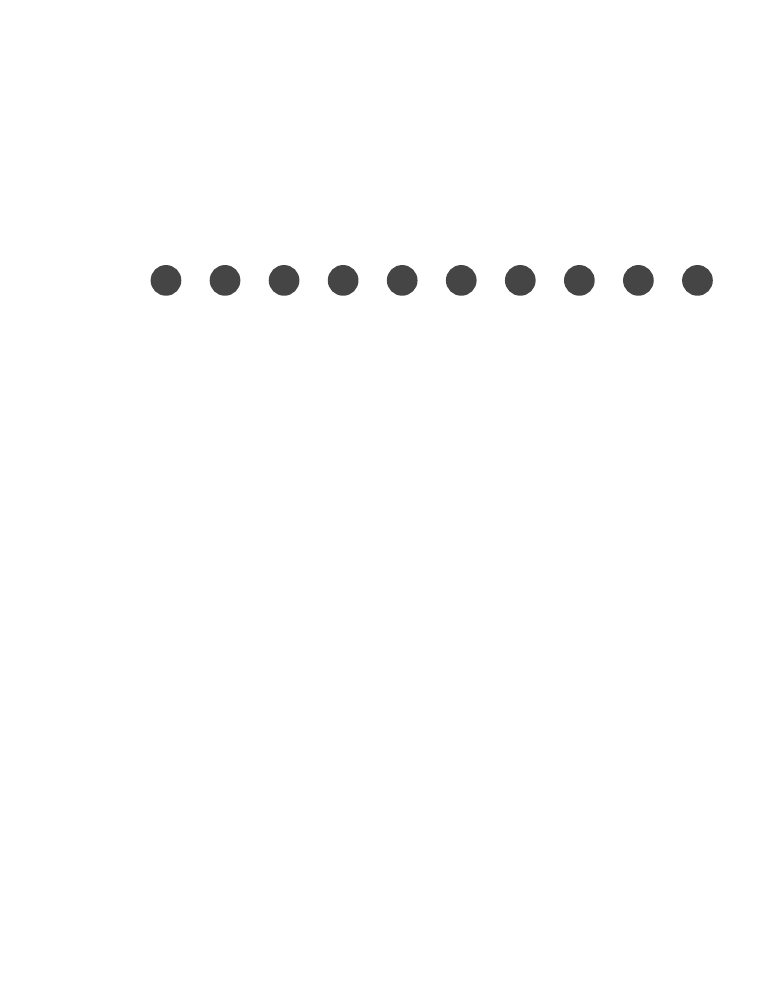 \\
\vspace{-2cm}
\end{wrapfigure}
Some examples of voting methods appear below.

\begin{example}
The \textbf{majority function} is the function
$$f(x)=\mathrm{sign}(x_{1}+x_{2}+\cdots+x_{n}).\qquad\qquad\qquad\qquad\qquad\qquad\qquad\qquad\qquad\qquad$$
If there are more $+1$ votes than $-1$ votes, then $f(x)=1$.  And if there are more $-1$ votes than $+1$ votes, then $f(x)=-1$.  That is, $f$ agrees with our usual notion of majority: the candidate receiving the most votes wins the election.  (To guarantee that someone wins the election, we could just assume that $n$ is odd, so that $f$ never takes the value $0$.)
\end{example}

\begin{example}
A \textbf{dictator function} is a function of the form
$$f(x)=x_{1}.\qquad\qquad\qquad\qquad\qquad\qquad\qquad\qquad\qquad\qquad$$
That is, the vote of the first person is the winner of the election.  In this way, $f$ agrees with our usual notion of dictator: all votes are ignored, except for one.  More generally, if $1\leq i\leq n$, a dictator is a function of the form $$f(x)=x_{i}.$$
\end{example}

\begin{example}
If $w_{1},\ldots,w_{n}$ are fixed real numbers, a \textbf{weighted majority function} on $n$ voters is a function of the form
$$f(x)=\mathrm{sign}(w_{1}x_{1}+w_{2}x_{2}+\cdots+w_{n}x_{n}).$$
If $w_{i}$ is large for some $1\leq i\leq n$, this corresponds to assigning more ``weight'' (i.e. more voting power, or more ``say'') to the $i^{th}$ voter.  And if $w_{i}$ is small, this corresponds to assigning less ``weight'' (i.e. less voting power, or less ``say'') to the $i^{th}$ voter.
\end{example}

\begin{figure}[h!]
 \def\svgwidth{\textwidth}
 \tiny
    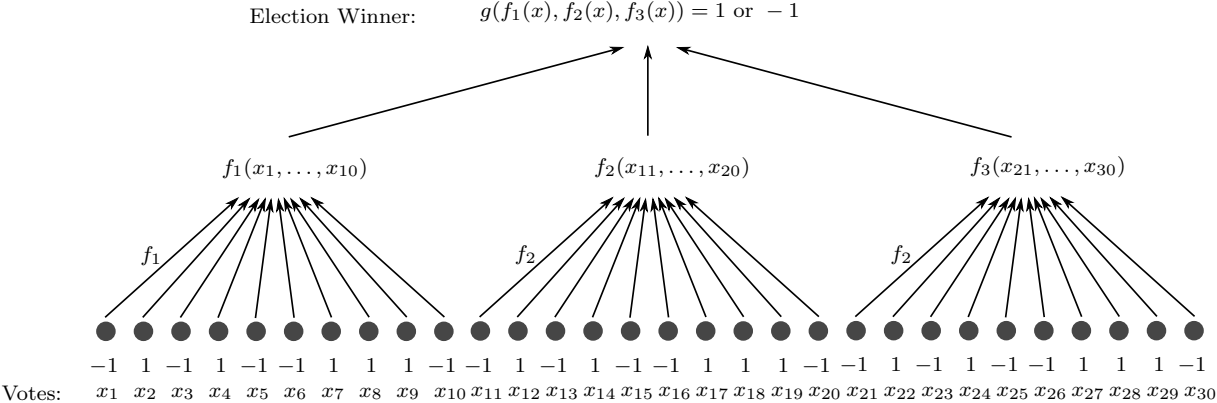 \\
    \normalsize
    \caption{An iterated majority function with $m=3$ ``states.''}\label{itmaj}
\end{figure}

\begin{example}\label{itmajex}
A two-layer \textbf{iterated majority function} is a function of the form
$$f(x)=g(f_{1}(x),f_{2}(x),\ldots,f_{m}(x)),$$
where $f_{1},f_{2},\ldots,f_{m}$ are each weighted majority functions on $n$ voters, and $g$ is a weighted majority function $m$ voters.

A two-layer iterated majority function is similar to an electoral college system with $m$ states.  The US electoral college system then corresponds to $m=51$.
\end{example}
\begin{remark}
In learning theory, the iterated majority function is sometimes called a two-layer neural network with boolean activation function.  The lines and nodes in Figure \ref{itmaj} are then interpreted as axons and neurons, respectively.
\end{remark}

In the ensuing discussion, it is more convenient to replace the Banzhaf power index of a voter with the (almost identical) notion of \textbf{influence} of a voter.

\begin{definition}[\embolden{Influences}]\label{infdef}
Let $f\colon\{-1,1\}^{n}\to\{-1,1\}$ be a voting method.  Let $1\leq i\leq n$ be an integer.  Define the \textbf{influence} of the $i^{th}$ voter on $f$, denote $\mathrm{Inf}_{i}(f)$, as
\begin{flalign*}
\mathrm{Inf}_{i}(f)
&=\frac{\mbox{\# of combinations of votes where the $i^{th}$ voter can change the election's outcome}}{\mbox{\# of combinations  of votes of all voters}}\\
&=\frac{\#\{(x_{1},\ldots,x_{n})\in\{-1,1\}^{n}\colon f(x_{1},\ldots,x_{n})\neq f(x_{1},\ldots,x_{i-1},-x_{i},x_{i+1},\ldots,x_{n})\}}{\#\{(x_{1},\ldots,x_{n})\in\{-1,1\}^{n}\}}.
\end{flalign*}
That is, $\mathrm{Inf}_{i}(f)$ is the probability that the $i^{th}$ voter can change the outcome of the election, when other voters are equally likely to vote for either candidate.
\end{definition}
\begin{example}
The numbers $b_{1},\ldots,b_{n}$ used to define the Banzhaf power indices are just the influences, multiplied by $2^{n}$.  For example, in the post-1965 UN Security council voting method $f\colon\{-1,1\}^{15}\to\{-1,1\}$ with $n=15$ voters,
$$\mathrm{Inf}_{1}(f)=\cdots=\mathrm{Inf}_{5}(f)=\frac{848}{2^{15}}\approx .0259,\qquad
\mathrm{Inf}_{6}(f)=\cdots=\mathrm{Inf}_{15}(f)=\frac{84}{2^{15}}\approx .00256.$$
Put another way, the Banzhaf power indices are the influences, multiplied by a number causing them to sum to $1$. 

\begin{table}[htbp!]\renewcommand{\arraystretch}{1.5}
\centering
\begin{tabular}{|l|p{4cm}|p{4cm}|c|}
\hline
Voting Method & Influence for \textit{Non}-Permament Member & Influence for Permament Member\\
\hline
Pre-1965 Rules & $\frac{5}{2^{11}}\approx .00244$ & $\frac{57}{2^{11}}\approx .0278$ \\
\hline
Post-1965 Rules & $\frac{84}{2^{15}}\approx .00256$ & $\frac{848}{2^{15}}\approx .0259$ \\
\hline
\end{tabular}
\caption{Influences for UN Security Council Voting Methods}
\end{table}

As above, we observe that a non-permanent member has a higher probability of affecting the outcome of a resolution in post-1965 rules.
\end{example}
\begin{example}\label{dictex}
When $f$ is a dictator function of the form $f(x)=x_{1}$, then the first voter can always change the outcome of the election, and the other voters cannot, so
$$I_{1}(f)=1,\quad I_{2}(f)=\cdots=I_{n}(f)=0.$$
When $f$ is a majority function $f(x)=\mathrm{sign}(x_{1}+\cdots+x_{n})$, then an application of Stirling's formula implies that for all $1\leq i\leq n$, $\lim_{n\to\infty}\sqrt{n}I_{i}(f)=\sqrt{\frac{2}{\pi}}$, i.e.
$$I_{1}(f)= I_{2}(f)=\cdots=I_{n}(f)=(1+o(1))\sqrt{\frac{2}{\pi}}\frac{1}{\sqrt{n}}.$$
To see this, note that if $n$ is even, recall that Stirling's Formula implies that
$$\binom{n}{n/2}=\frac{n!}{[(n/2)]!}=(1+o(1))\frac{1}{\sqrt{2\pi}}\frac{\sqrt{n}}{n/2}2^{n}=(1+o(1))2^{n}\frac{1}{\sqrt{n}}\sqrt{\frac{2}{\pi}}.$$
Therefore, $\lim_{n\to\infty}\sqrt{n}I_{i}(f)=\lim_{n\to\infty}(1+o(1))\frac{1}{\sqrt{2\pi}}\frac{2^{n}}{2^{n}}=\frac{1}{\sqrt{2\pi}}$, for all $1\leq i\leq n$.
\end{example}
Perhaps it is a compelling reason to vote in a majority election with one hundred million voters when your probability of changing the election's outcome is around $1$ in ten thousand.

\section{Adversarial Corruption in Voting}

\subsection{Two Candidates}

\begin{wrapfigure}{r}{0.52\textwidth}
   \def\svgwidth{.5\textwidth}
    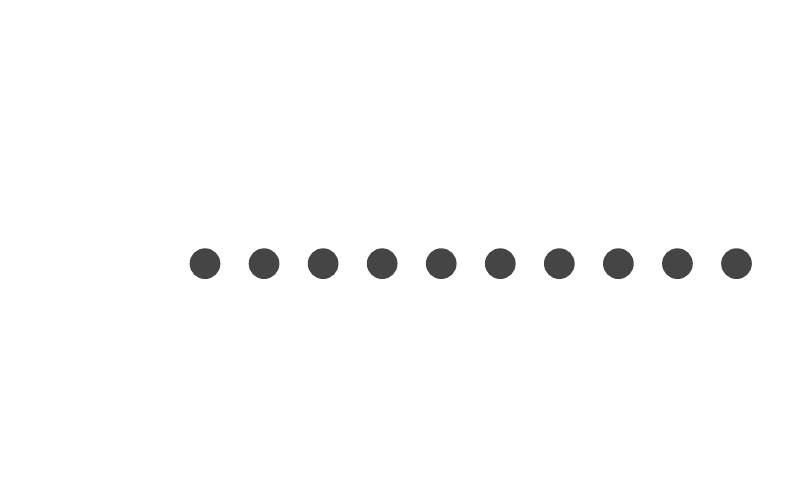 \\
    \vspace{-.5cm}
\end{wrapfigure}

Suppose $n$ people cast their votes in an election between two candidates.  Then, suppose an adversary found a way to change several of the votes.  By changing some votes, the adversary attempts to change the outcome of the election.  Suppose also that the voting method $f\colon\{-1,1\}^{n}\to\{-1,1\}$ is balanced in the following sense.

\begin{definition}[\embolden{Balanced Voting Method}]\label{baldef}
Let $f\colon\{-1,1\}^{n}\to\{-1,1\}$ be a voting method.  We say $f$ is \textbf{balanced} if each of the two candidates has an equal chance of winning the election.  That is, the number of combinations of votes where candidate $1$ wins is equal to the number of combinations of votes where candidate $-1$ wins.
\end{definition}

For example, dictator functions and the majority function are balanced.

\begin{question}
What balanced voting method is most resilient to adversarial changes to votes?

That is, if $k\geq1$ votes can be changed by the adversary, what is the least number of combinations of votes (of all voters) such that the adversary can change the election's outcome?
\end{question}

In a dictatorship, e.g. $f(x_{1},\ldots,x_{n})=x_{1}$, changing the first vote changes the outcome of the election, so this voting method is not at all resilient to adversarial changes.  Similarly, a voting method that is only a function of a small set of voters (sometimes called a junta) will probably not be resilient to adversarial changes to votes.  It turns out that the majority function is the balanced voting method most resilient to adversarial changes; we thank Daniel Kane for telling us the following argument.

\begin{prop}[\embolden{Adversarial Optimality of Majority}]\label{advmaj}
Let $n$ be an odd positive integer and let $k$ be an integer satisfying $1\leq k\leq n$.  After the votes have been cast, suppose an adversary can change $k$ votes in an election between two candidates with $n$ voters.  Then among all balanced voting methods, the majority function has the least number of combinations of votes where the election's outcome can be altered by the adversary.
\end{prop}

Before beginning the proof, we introduce some notation.  For any $x=(x_{1},\ldots,x_{n})\in\R^{n}$, denote the $\ell_{0}$ ``norm'' of $x$ by $\vnorm{x}_{0}\colonequals\#\{1\leq i\leq n\colon x_{i}\neq 0\}$.  (This quantity is not a norm since $\vnorm{tx}_{0}=\vnorm{x}_{0}$ for any $t\neq0$.)  Let $S\subset\{-1,1\}^{n}$.  For any integer $k\geq1$, we denote the distance $k$ neighborhood of $S$ by
\begin{equation}\label{gammadef}
\Gamma_{k}(S)\colonequals\{x\in\{-1,1\}^{n}\colon \exists\, y\in S\,\,\mathrm{such}\,\,\mathrm{that}\,\, \vnorm{x-y}_{0}\leq k\}.
\end{equation}
Then $\Gamma_{k}(S)$ is the set of possible votes that can be obtained by changing at most $k$ votes from a given $y\in S$.  For any $k\geq0$, let $B_{k}\subset\{-1,1\}^{n}$ be a distance $k$ neighborhood of one ``half'' of the hypercube:
\begin{equation}\label{bkdef}
B_{k}\colonequals\Gamma_{k}\big(\{(y_{1},\ldots,y_{n})\in\{-1,1\}^{n}\colon y_{1}+\cdots+y_{n}\geq 0\}\big).
\end{equation}

\begin{wrapfigure}{r}{0.53\textwidth}
   \def\svgwidth{.47\textwidth}
\begingroup%
  \makeatletter%
  \providecommand\color[2][]{%
    \errmessage{(Inkscape) Color is used for the text in Inkscape, but the package 'color.sty' is not loaded}%
    \renewcommand\color[2][]{}%
  }%
  \providecommand\transparent[1]{%
    \errmessage{(Inkscape) Transparency is used (non-zero) for the text in Inkscape, but the package 'transparent.sty' is not loaded}%
    \renewcommand\transparent[1]{}%
  }%
  \providecommand\rotatebox[2]{#2}%
  \newcommand*\fsize{\dimexpr\f@size pt\relax}%
  \newcommand*\lineheight[1]{\fontsize{\fsize}{#1\fsize}\selectfont}%
  \ifx\svgwidth\undefined%
    \setlength{\unitlength}{293.72746265bp}%
    \ifx\svgscale\undefined%
      \relax%
    \else%
      \setlength{\unitlength}{\unitlength * \real{\svgscale}}%
    \fi%
  \else%
    \setlength{\unitlength}{\svgwidth}%
  \fi%
  \global\let\svgwidth\undefined%
  \global\let\svgscale\undefined%
  \makeatother%
  \begin{picture}(1,0.93069383)%
    \lineheight{1}%
    \setlength\tabcolsep{0pt}%
    \put(0,0){\includegraphics[width=\unitlength,page=1]{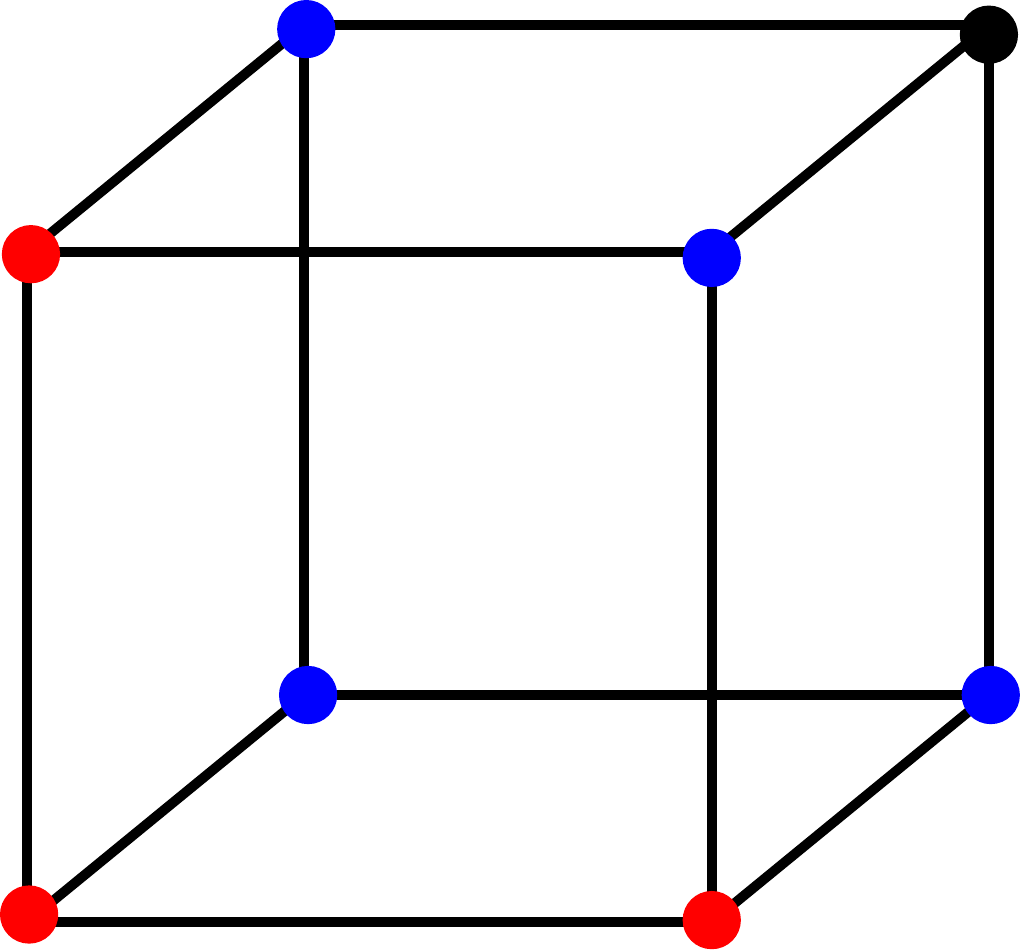}}%
    \put(0.14773172,0.05471559){\color[rgb]{0,0,0}\makebox(0,0)[lt]{\lineheight{1.25}\smash{\begin{tabular}[t]{l}$\textcolor{red}{S}$\end{tabular}}}}%
    \put(0.31151324,0.56539297){\color[rgb]{0,0,0}\makebox(0,0)[lt]{\lineheight{1.25}\smash{\begin{tabular}[t]{l}The blue region\\and red region\\together \\form $\Gamma_{1}(S)$\end{tabular}}}}%
  \end{picture}%
\endgroup%
 \\
    \vspace{.2cm}
\end{wrapfigure}

The key geometric fact used to prove Proposition \ref{advmaj} is:
\begin{theorem}[\embolden{Harper's Inequality/ Hypercube Vertex Isoperimetric Inequality}]\label{harper}
Let $S\subset\{-1,1\}^{n}$.  Let $k\geq0$.  Assume that
$$\abs{S}\geq\abs{B_{k}}.\qquad\qquad\qquad\qquad\qquad\qquad\qquad\qquad\qquad\qquad$$
Then
$$\abs{\Gamma_{1}(S)}\geq\abs{\Gamma_{1}(B_{k})}.\qquad\qquad\qquad\qquad\qquad\qquad\qquad\qquad\qquad\qquad$$
\end{theorem}

\begin{proof}[Proof of Proposition \ref{advmaj}]
We induct on $k$.  Let $f\colon\{-1,1\}^{n}\to\{-1,1\}$ be the majority function, and let $g\colon\{-1,1\}^{n}\to\{-1,1\}$ be another balanced voting method.  Let $S\colonequals\{x\in\{-1,1\}^{n}\colon g(x)=1\}$ be the set of votes where candidate $1$ wins the election, when $g$ is the voting method used to run the election.  Note that $B_{0}=\{x\in\{-1,1\}^{n}\colon f(x)=1\}$.  Since $f$ and $g$ are balanced, $\abs{S}=\abs{B_{0}}=2^{n-1}$.  So, Harper's Inequality, Theorem \ref{harper}, implies that
\begin{equation}\label{one1}
\abs{\Gamma_{1}(S)}\geq\abs{\Gamma_{1}(B_{0})}.
\end{equation}
\begin{equation}\label{one2}
\abs{\Gamma_{1}(S)}-\abs{S}\geq\abs{\Gamma_{1}(B_{0})}-\abs{B_{0}}.
\end{equation}
The same inequality holds also when $S\colonequals\{x\in\{-1,1\}^{n}\colon g(x)=-1\}$.  Taken together, we conclude that the number of combinations of votes for which the outcome of the election can be altered with one adversarial vote change is smallest for the majority vote $f$ (since $f$ corresponds to the right side of \eqref{one2}).  The case $k=1$ therefore follows by \eqref{one2}.

We now proceed with the inductive step.  By the inductive hypothesis, if $S\colonequals\{x\in\{-1,1\}^{n}\colon g(x)=1\}$ or if $S\colonequals\{x\in\{-1,1\}^{n}\colon g(x)=-1\}$, we have
$$\abs{\Gamma_{k}(S)}-\abs{S}\geq\abs{\Gamma_{k}(B_{0})}-\abs{B_{0}}.$$
That is, $\abs{\Gamma_{k}(S)}\geq\abs{\Gamma_{k}(B_{0})}=\abs{B_{k}}$.  We need to prove the case $k+1$.  This again follows by Harper's Inequality, Theorem \ref{harper}, since
$$\abs{\Gamma_{k+1}(S)}\stackrel{\eqref{gammadef}}{=}\abs{\Gamma_{1}(\Gamma_{k}(S))}\geq\abs{\Gamma_{1}(B_{k})}\stackrel{\eqref{bkdef}}{=}\abs{B_{k+1}},$$
Therefore, when $S\colonequals\{x\in\{-1,1\}^{n}\colon g(x)=1\}$ or $S\colonequals\{x\in\{-1,1\}^{n}\colon g(x)=-1\}$,
\begin{equation}\label{one3}
\abs{\Gamma_{k+1}(S)}-\abs{S}\geq \abs{B_{k+1}}-\abs{B_{0}}.
\end{equation}
That is, the number of votes for which the outcome of the election can be altered with $k+1$ adversarial vote changes is smallest for the voting method $f$ (since the majority vote $f$ corresponds to the right side of \eqref{one3}).  The inductive step and the proof are complete.
\end{proof}

For some related observations for ranked choice voting, see e.g. \cite[Lemma 3.3]{mossel13}.

Proposition \ref{advmaj} can easily be extended to unbalanced voting methods.  To state such a result, let $t$ be a real number and define a \textbf{majority function with threshold $t$} to be a function of the form
$$\mathrm{Maj}_{n,t}(x)=\mathrm{sign}(x_{1}+x_{2}+\cdots+x_{n}-t),\qquad\forall\,x=(x_{1},\ldots,x_{n})\in\{-1,1\}^{n}.$$
Also, we say that two voting methods $f,g\colon\{-1,1\}^{n}\to\{-1,1\}$ \textbf{have the same balance} if the number of combinations of votes resulting in candidate $1$ winning are the same for each voting method, i.e.
$$\#\{(x_{1},\ldots,x_{n})\in\{-1,1\}^{n}\colon f(x)=1\}=\#\{(x_{1},\ldots,x_{n})\in\{-1,1\}^{n}\colon g(x)=1\}.$$
For example, the majority function with threshold $t=0$ and the majority function with threshold $t=1$ do not have the same balance.

\begin{prop}[\embolden{Adversarial Optimality of Majority, Unbalanced Case}]\label{advmaj2}
Let $n$ be an odd positive integer and let $k$ be an integer satisfying $1\leq k\leq n$.  After the votes have been cast, suppose an adversary can change $k$ votes in an election between two candidates with $n$ voters.  Let $f$ be a majority function with threshold $t$, where $t$ is an even integer.  Let $g$ be another voting method such that $f$ and $g$ have the same balance.  Then the number of combinations of votes where the election's outcome can be altered by the adversary is lesser for $f$ than for $g$.
\end{prop}

\subsection{More than Two Candidates}

It would be desirable to have an analogue of Proposition \ref{advmaj} for voting methods with more than two candidates.  Such a result might require a version of Harper's Inequality, Theorem \ref{harper}, for multiple sets.  It is unclear if such an inequality can be proven

\subsection{Additional Comments}

Proposition \ref{advmaj} can be strengthened slightly, so that a voting method that is ``close'' to being as resilient as majority must itself be ``close'' to majority.  Instead of applying Theorem \ref{harper}, one instead uses a stronger version, such as \cite{keevash18}.

The majority function is known to be optimal in various senses.  For example, the majority function maximizes the number of votes that agree with the outcome of the election \cite[Theorem 2.33]{odonnell14}.  Apparently Rousseau argued this was an ideal choice for a voting method in 1762 in ``Du contrat social.''  Theorem \ref{mis} below, the Majority is Stablest Theorem, also characterizes the majority function as being the most stable to random corruption in votes, among a reasonable class of voting methods.

For more background on social choice theory, see e.g. \cite[Chapter 2]{odonnell14}, \cite{odonnell14b}, \cite[Section 3]{kalai18}.

\section{Independent Random Corruption of Votes}

\begin{wrapfigure}{r}{0.5\textwidth}
   \def\svgwidth{.5\textwidth}
    \input{vote6.pdf_tex} \\
    \vspace{-1cm}
\end{wrapfigure}

In Proposition \ref{advmaj}, we showed that the majority function is the most stable voting method to adversarial corruption.  The majority function is also most stable when votes are corrupted randomly, as shown below.

\begin{theorem}[\embolden{Majority is Stablest, Informal Version}, {\cite[Theorem 4.4]{mossel10}}]\label{misinf}
Suppose we run an election with a large number $n$ of voters and two candidates.  In this election, voters are modelled to have the following random behavior:
\begin{itemize}
\item[(i)] Voters cast their votes randomly, independently, with equal probability of voting for either candidate.
\item[(ii)] Each voter has a small influence on the outcome of the election.  (That is, all influences from Definition \ref{infdef} are small.)
\end{itemize}
Then the majority function is the balanced voting method that best preserves the outcome of the election, when votes have been corrupted independently.
\end{theorem}

The definition of ``best'' here is intentionally vague.  We will define ``best'' to mean: maximizing noise stability, as defined below in Definition \ref{nsdef}.  Also, the probability of each vote being changed (corrupted) should be less than $1/2$ in Theorem \ref{misinf}.  Otherwise the majority preferences of the electorate are reversed upon corruption.

Some remarks concerning the sensibility of the assumptions of Theorem \ref{misinf} now follow.

\begin{itemize}
\item Suppose we completely ignore the votes, and just declare that the first candidate wins.  This voting method is as stable to vote corruption as one can imagine, since any amount of corruption in votes cannot change the outcome of the election.  Since this voting method is certainly undemocratic and uninteresting, some assumption in Theorem \ref{misinf} must eliminate it.  And indeed, this voting method is not balanced, so Theorem \ref{misinf} ignores it.  This voting method corresponds to a constant function $f$.
\item As we saw in Example \ref{dictex}, a dictator function has one large influence, and the remaining voters have no influence on the election's outcome.  Consequently, the dictator voting method is quite stable to independently random changes to votes, since changing the votes of the non-dictators has no effect on the election's outcome.  So, as in the previous example, the dictator function is rather stable to vote corruption for a rather uninteresting reason.  We therefore eliminate dictator functions from consideration by imposing the democratic assumption $(ii)$ that each voter has a small influence on the outcome of the election.
\end{itemize}

\subsection{Two Candidates}

In this section, we will formalize the assumptions in Theorem \ref{misinf}, resulting in the formal version of the Majority is Stablest Theorem \ref{mis}.

\begin{assumption}[\embolden{Voter Assumptions}]\label{voteas}
\hfill
\begin{itemize}
\item There are $n$ voters denoted $\{1,\ldots,n\}$.  There are two candidates denoted $-1$ and $1$.
\item For any $1\leq i\leq n$, the $i^{th}$ voter casts a single \textit{random} vote $X_{i}$ taking the value $-1$ or $1$.  (In particular, we are not dealing with ranked voting methods)
\item The votes $(X_{1},\ldots,X_{n})$ are independent, identically distributed (i.i.d.) random variables.  That is, voters are modelled as independent decision makers with the same probabilities of voting for either candidate.
\end{itemize}
The \textbf{voting method} $f$ is a function $f\colon\{-1,1\}^{n}\to\{-1,1\}$.  If the votes are $(X_{1},\ldots,X_{n})$, then the winner of the election is $f(X_{1},\ldots,X_{n})$.
\end{assumption}

\begin{remark}
One could argue that the voter assumptions are not realistic, since e.g. a small group of friends will most likely share similar views, read similar news items, etc., so that their decisions are not truly independent.  On the other hand, modeling a large number of voters to be independent individuals is somewhat plausible, from an aggregate perspective.
\end{remark}

\begin{assumption}[\embolden{Voter Corruption Assumptions}]\label{corras}
Let $0\leq\rho\leq1$.  Suppose we are given the votes $X_{1},\ldots,X_{n}$ of $n$ voters choosing between $2$ candidates.  The corrupted votes $Y_{1},\ldots,Y_{n}$ are defined as follows.
\begin{itemize}
\item The \textbf{corrupted votes} $Y_{1},\ldots,Y_{n}$ are independent, identically distributed (i.i.d.) random variables.
\item For each $1\leq i\leq n$, if $X_{i}=x_{i}\in\{-1,1\}$, then with probability $1-\rho$, $Y_{i}$ is a uniformly random element of $\{-1,1\}$, and with probability $\rho$, $Y_{i}=x_{i}$.
\end{itemize} 
\end{assumption}
\begin{remark}
When $\rho=1$, $Y_{i}=X_{i}$ for all $1\leq i\leq n$, i.e. no vote corruption has occurred.  When $\rho$ is close to $1$, $Y_{1}$ is almost the same as $X_{1}$, i.e. $X_{1}$ and $Y_{1}$ are strongly correlated, and a small amount of vote corruption has occurred.

When $\rho=0$, the votes $(X_{1},\ldots,X_{n})$ and $(Y_{1},\ldots,Y_{n})$ are independent of each other, i.e. the corrupted votes $(Y_{1},\ldots,Y_{n})$ have been so scrambled that they have no dependence (or correlation) with the original votes $(X_{1},\ldots,X_{n})$.
\end{remark}

\noindent\textbf{Notation.}  We denote the original (random) votes cast in the election as $X=(X_{1},\ldots,X_{n})$, and we denote the corrupted votes as $Y=(Y_{1},\ldots,Y_{n})$.

Recall that the voting method $f$ takes the value $1$ or $-1$, according to which candidate ($1$ or $-1$) won the election.  So, if the winner of the election $f(X)$ is the same as the winner of the election with corrupted votes $f(Y)$, then
$$f(X)f(Y)=1.$$
On the other hand, if the winner of the election $f(X)$ is different than the winner of the election with corrupted votes $f(Y)$, then
$$f(X)f(Y)=-1.$$
So, the voting method that has the largest average value of
$$f(X)f(Y)$$
will be the most stable on average to random vote corruption.  This observation motivates the following definition.

\begin{definition}[\embolden{Noise Stability}]\label{nsdef}
Let $f\colon\{-1,1\}^{n}\to\{-1,1\}$ be a voting method.  The \textbf{noise stability} of $f$ with correlation parameter $0\leq\rho\leq1$ is
$$S_{\rho}(f)\colonequals \E f(X)f(Y).$$
Here $\E$ denotes expected value, or average value, with respect to the random variables $X=(X_{1},\ldots,X_{n}),Y=(Y_{1},\ldots,Y_{n})$ defined in Assumptions \ref{voteas} and \ref{corras}.
\end{definition}

\begin{remark}\label{nsprobrk}
The probability that the election's outcome stays the same after vote corruption has occurred is $\frac{1}{2}(1+S_{\rho}(f))$.
\end{remark}

\subsubsection{Unbiased Case}

Theorem \ref{misinf} can be restated as: the majority function maximizes noise stability, among a reasonable class of voting methods.

In the Theorem below, we denote the Majority function as $\mathrm{Maj}_{n}\colon\{-1,1\}^{n}\to\{-1,1\}$, so that
$$\mathrm{Maj}_{n}(x_{1},\ldots,x_{n})=\mathrm{sign}(x_{1}+\cdots+x_{n}),\qquad\mathrm{for}\,\,\mathrm{all}\,\,(x_{1},\ldots,x_{n})\in\{-1,1\}^{n}.$$

\pagebreak

\begin{wrapfigure}{r}{0.5\textwidth}
\vspace{-1cm}
\centering
    \,\,\includegraphics[width=.45\textwidth]{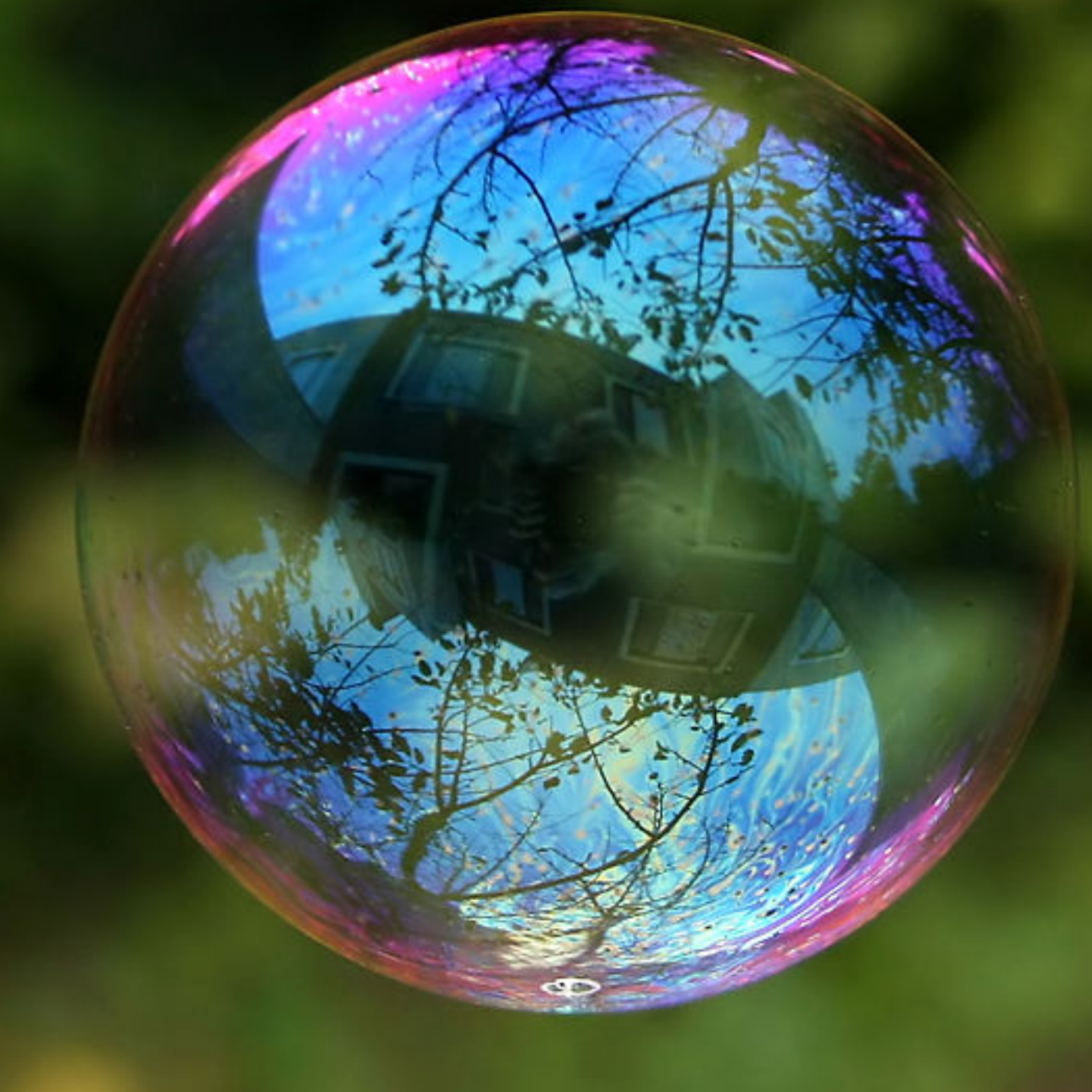} \\
    \caption[blah]{The proof of Theorem \ref{mis} is related to the fact that soap bubbles take a spherical shape.  A soap bubble\footnotemark\, encloses a fixed volume of air, and it minimizes its surface area. The majority function has an analogous optimality property.  We will discuss this connection more in Section \ref{otherapp}.}
      \vspace{-1cm}
\end{wrapfigure}

For simplicity, we first state the balanced case of the Theorem.  That is, we make the assumption that the random votes $X_{1},\ldots,X_{n}$ are each uniformly distributed in $\{-1,1\}$.  So, e.g. $X_{1}=1$ with $1/2$ probability, and $X_{1}=-1$ with $1/2$ probability.

\begin{theorem}[\embolden{Majority is Stablest, Formal Version}, {\cite[Conjecture 1.1]{mossel10}}]\label{mis}
Let $0\leq\rho\leq1$ and let $\epsilon>0$.  Then there exists $\tau>0$ such that, if $f\colon\{-1,1\}^{n}\to\{-1,1\}$ satisfies $\E f(X)=0$ and $\mathrm{Inf}_{i}(f)\leq\tau$ for all $1\leq i\leq n$, then
$$S_{\rho}(f)\leq\lim_{n\to\infty} S_{\rho}(\mathrm{Maj}_{n})+\epsilon=\frac{2}{\pi}\sin^{-1}(\rho)+\epsilon.\qquad\qquad\qquad\qquad\qquad\qquad\qquad\qquad\qquad\qquad$$
\end{theorem}

The assumption $\E f(X)$ says that $f$ is balanced according to Definition \ref{baldef}, and the assumption $\max_{1\leq i\leq n}\mathrm{Inf}_{i}(f)\leq\tau$ corresponds to part (ii) of Theorem \ref{misinf}.

\footnotetext{Picture taken from https://commons.wikimedia.org/wiki/File:Reflection\underline{ }in\underline{ }a\underline{ }soap\underline{ }bubble\underline{ }edit.jpg}

\subsubsection{Biased Case}

The assumption in Theorem \ref{mis} that the votes are uniformly distributed in $\{-1,1\}$ can be relaxed, as we now describe.  Let $0<p<1$.  Let $X_{1},\ldots,X_{n}$ be independent identically distributed random variables where $\P(X_{i}=1)=1-\P(X_{i}=-1)=p$ for all $1\leq i\leq n$.

\begin{theorem}[\embolden{Majority is Stablest, Formal, Biased Case}, {\cite[Theorem 4.4]{mossel10}}]\label{misg}
Let $0\leq\rho\leq1$.  Let $-1\leq\mu\leq 1$.  Let $t=t_{n}\in\R$ such that $\abs{\E \mathrm{Maj}_{n,t}(X)-\mu}=\min_{t'\in\R}\abs{\E\mathrm{Maj}_{n,t'}-\mu}$.  Let $\tau>0$ and let $f\colon\{-1,1\}^{n}\to\{-1,1\}$ satisfy $\E f(X)=\mu$ and $\mathrm{Inf}_{i}(f)\leq\tau$ for all $1\leq i\leq n$, then
$$S_{\rho}(f)\leq\lim_{n\to\infty} S_{\rho}(\mathrm{Maj}_{n,t_{n}})+O_{p,1-\rho}\Big(\frac{\log\log(1/\tau)}{\log(1/\tau)}\Big).$$
\end{theorem}

For an even more general version of Theorem \ref{misg}, see \cite[Theorem 4.4]{mossel10}

\subsection{More than Two Candidates}

In this section, we consider elections between $k\geq3$ candidates, where each of $n$ voters casts a single vote for a single candidate.

Theorem \ref{misg} (and its generalizations such as \cite[Theorem 4.4]{mossel10}) essentially completely characterize majority functions as the most stable to independently random corruption of votes, \textit{when the election has only two candidates}.  Unfortunately, analogous statements for three or more candidates seem harder to prove.  With more than two candidates, a suitable replacement for the majority is the plurality function.  In a plurality election, the candidate with the most votes wins the election.

\begin{wrapfigure}{r}{0.52\textwidth}

\vspace{-.5cm}
\centering
    \includegraphics[width=.26\textwidth]{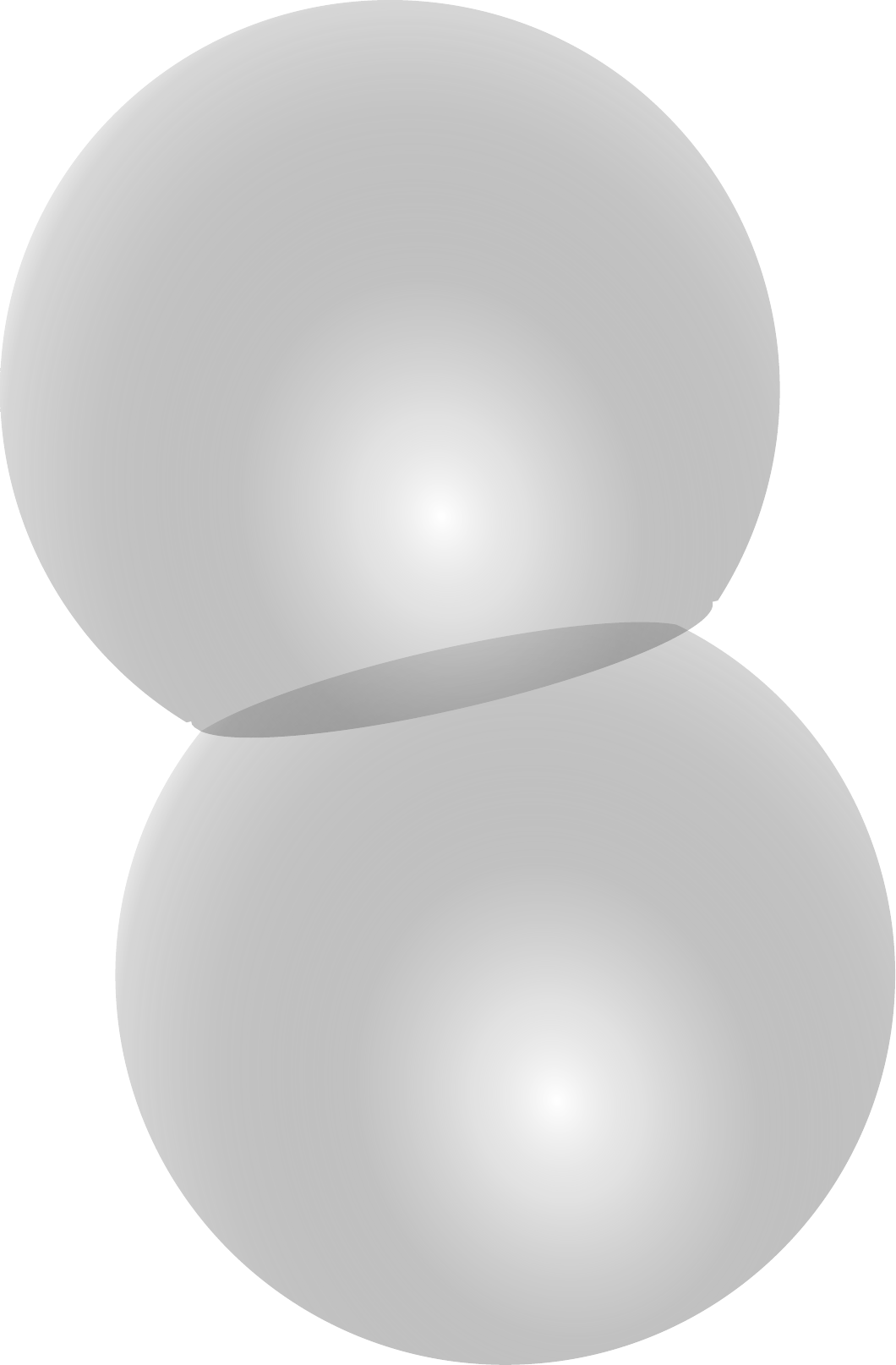} \\

    \caption{Generalizing Theorem \ref{misg} to elections with three candidates is related to: proving that two joint soap bubbles take the pictured ``double-bubble'' shape.  Two soap bubbles enclose two separate and fixed volumes of air, and they minimize their total surface area \cite{hutchings02}. The plurality function should have an analogous optimality property.  We will discuss this connection more in Section \ref{otherapp}.}
    \label{doublepic}
%
%
    \includegraphics[width=.3\textwidth]{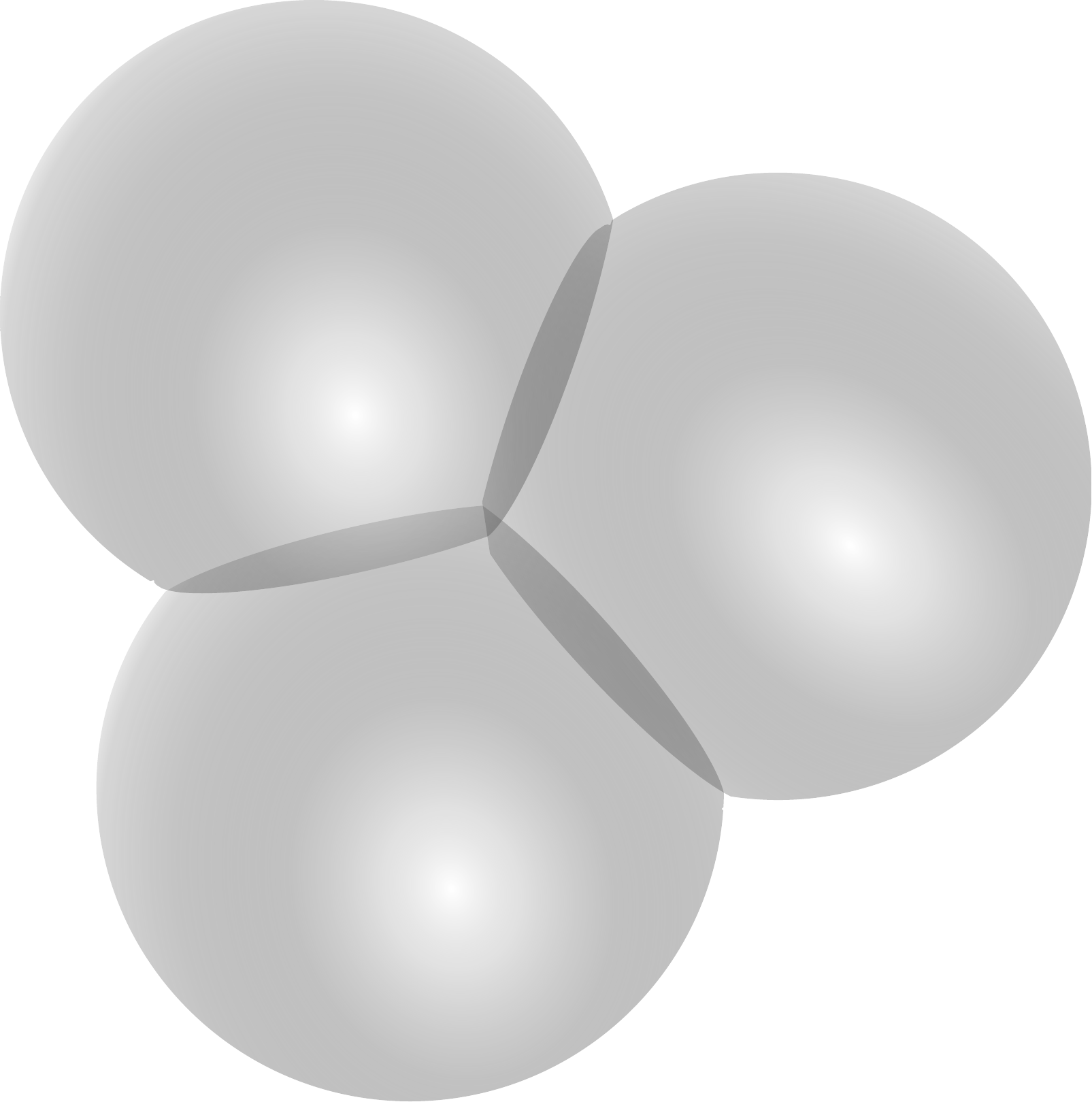} \\

    \caption{Three soap bubbles that have collided take the shape shown here.}\label{tripic}
        \vspace{-2.2cm}
\end{wrapfigure}

It was conjectured \cite{khot07,isaksson11} that the plurality function is the balanced voting method that is most stable to independent, random vote corruption

\begin{conj}[\embolden{Plurality is Stablest, Informal Version}, {\cite{khot07}, \cite[Conjecture 1.9]{isaksson11}}]\label{pisinf}
Suppose we run an election with a large number $n$ of voters and $k\geq3$ candidates.  In this election, voters are modelled to have the following random behavior:
\begin{itemize}
\item Voters cast their votes randomly,\\ independently, with equal probability of voting for each candidate.
\item Each voter has a small influence on\\ the outcome of the election.
\end{itemize}
Then the plurality function is the balanced\\ voting method that best preserves the out-\\ come of the election, when votes\\ have been corrupted independently.
\end{conj}

In the case that the probability of vote corruption is small ($\rho$ is close to $1$), we proved the first known case of Conjecture \ref{pisinf} in \cite{heilman18b}, culminating a series of previous works. Conjecture \ref{pisinf} for all parameters $0<\rho<1$ is still open.  Unlike the case of the Majority is Stablest (Theorem \ref{misg}), Conjecture \ref{pisinf} \textit{cannot hold} when the candidates have unequal chances of winning the election \cite{heilman14}.  This realization is an obstruction to proving Conjecture \ref{pisinf}.  It suggested that proof methods for Theorem \ref{misg} cannot apply to Conjecture \ref{pisinf}.  Indeed, calculus of variations methods have emerged as a promising avenue for proving Conjecture \ref{pisinf}, when the candidates have equal chances of winning the election.

%

 \subsection{Additional Comments}

Discrete Fourier analysis often plays a prominent role in noise stability and voting.  The surveys \cite{odonnell14b,khot10b} and book \cite{odonnell14} describe the interconnectedness of these topics.

\pagebreak

\begin{wrapfigure}{r}{0.5\textwidth}
\centering
    \includegraphics[width=.3\textwidth]{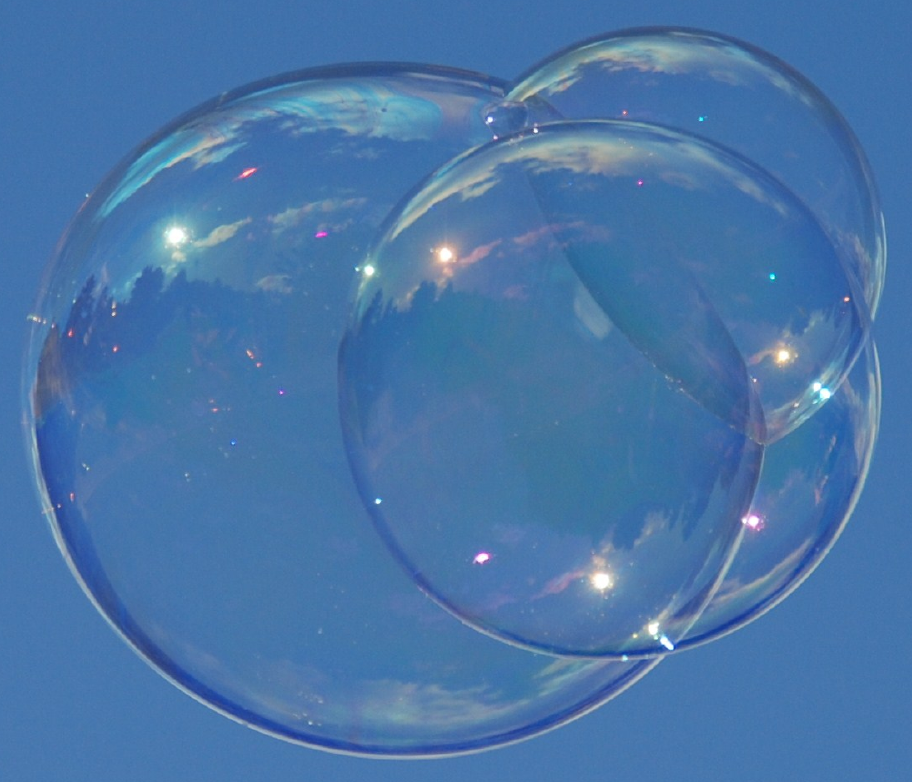} \\
    \caption[blah]{Three soap bubbles that have collided take the shape shown here\footnotemark.}
    \vspace{-1cm}
\end{wrapfigure}

We have not focussed much on ranked choice voting methods.  For more on this topic, see e.g. \cite{mossel13} or the comprehensive works \cite{arrow02,brandt16}.

\begin{question}\label{rankedq}
Is it possible to state a sensible version of the Plurality is Stablest Conjecture \ref{pisinf} for ranked choice voting methods?
\end{question}

\begin{wrapfigure}{r}{0.5\textwidth}
\vspace{-.5cm}
\centering
    \includegraphics[width=.25\textwidth]{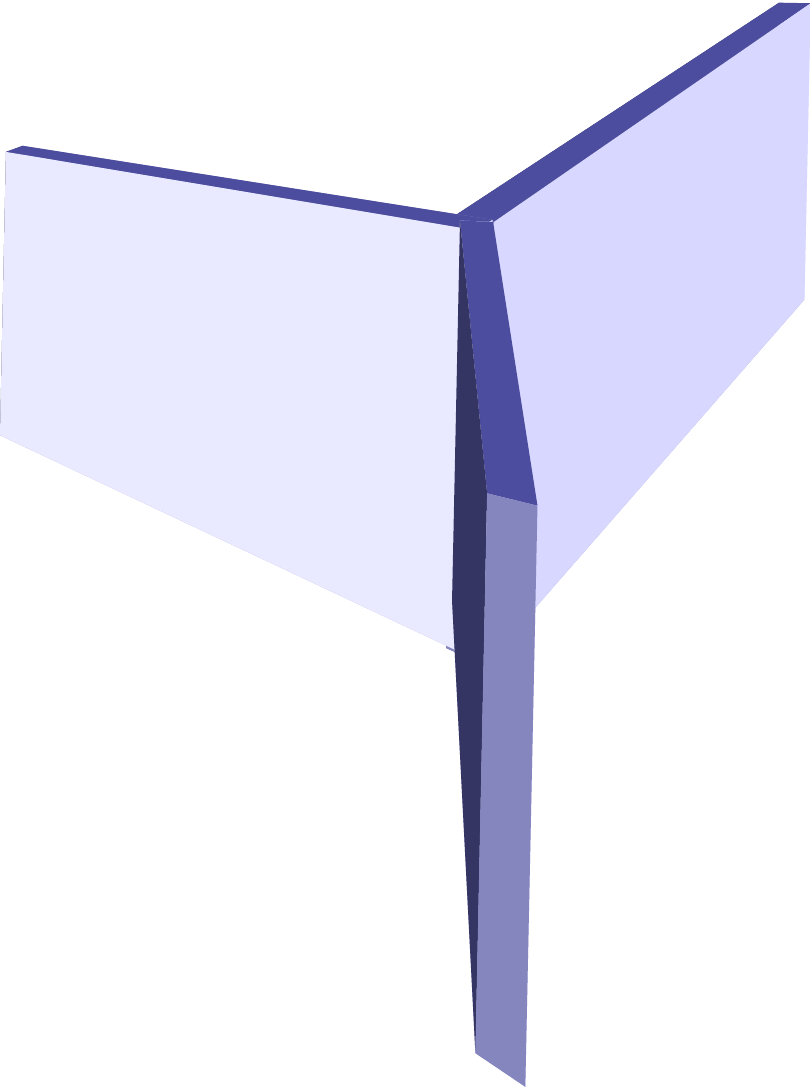} \\
    \caption{The partition of Euclidean space with three regions of fixed Gaussian volume and minimal total Gaussian surface area takes this shape \cite{milman18b,heilman18b}.}
     \vspace{-1.3cm}
\end{wrapfigure}

In ranked choice voting, each voter provides a ranked list of the candidates.  Suppose a voting method is then a function only of the pairwise comparisons of each candidate, as in Table \ref{table2}.  Suppose then that each of these pairwise comparisons is independently corrupted.  Then one possible answer to Question \ref{rankedq} says that the Plurality of the pairwise comparisons is most stable to this kind of vote corruption.  Taking the Plurality of pairwise comparisons is known as the Second Order Copeland voting method.  So, one could argue that this method is most stable to vote corruption.  However, under other models of vote corruption, it is not clear what the ``best'' ranked choice voting method should be.

\footnotetext{Picture taken from https://www.flickr.com/photos/sm/2603411754/sizes/o/}

\section{Brief Discussion of US Electoral College}

The US Electoral College system is similar but not identical to the two-tier majority function described in Example \ref{itmajex} with $m=51$ equal-sized ``states.''  Suppose we run an election between two candidates, where $g$ is a two-tier majority function with $m=51$, and $f$ is the usual majority function with $n$ a large odd number of voters.  We already know from Theorem \ref{misg} that the majority function is more stable to vote corruption that the electoral college system.  But how much more stable is it?

We consider the noise stability $S_{\rho}$ of each of these voting methods where $\rho=1-2\epsilon$ and $\epsilon>0$ is small with $51<1/\epsilon<n$.  That is, the probability of each vote being corrupted is small.  Then \cite[p. 9]{odonnell08}
$$S_{1-2\epsilon}(f)\approx 1-\frac{4}{\pi}\sqrt{\epsilon},\qquad S_{1-2\epsilon}(g)\approx 1-2(\frac{2}{\pi})^{3/2}\sqrt{51}\sqrt{\epsilon}.$$
So, by Remark \ref{nsprobrk}, the probability that vote corruption changes the election's outcome is about 5.7 times greater for the electoral college, than for majority.  Here we used
$$\frac{2(2/\pi)^{3/2}\sqrt{51}}{4/\pi}\approx 5.698035\ldots$$

Computer simulation similarly shows that, when the size of each ``state'' defining the two-tier majority function agrees with their size from the 2010 census, then the probability that vote corruption changes the election's outcome is more than 4 times greater for the electoral college, than for majority.

Strictly speaking, the US Electoral College system is not a two-tier majority function.  Each state (except for Maine and Nebraska) runs its own plurality vote, so that the candidate winning the most votes in that state wins that particular state.  Then each state's winner is entered into a nationwide weighted majority vote.  If no candidate wins this weighted majority vote, then the House of Representatives chooses the president, with one vote for each of the 50 states.  In an election between two candidates, a perfect tie in the electoral college is unlikely, i.e. it is unlikely for the House of Representatives to choose the president.  Moreover, since Maine and Nebraska are small states, their chance of changing the outcome in the electoral college is small.  So, the probability that vote corruption changes the U.S. presidential election's outcome is still more than 4 times greater for the electoral college, than for majority.

The integer weight of each state in the nationwide majority vote is equal to the number of national congressional representatives in each state (with a weight of three given to Washington D.C.)  Consequently, each state has a minimum weight of $3$ in the nationwide majority vote (i.e. the electoral college vote).  The apportionment of members to the House of Representatives is a nontrivial task, since the ratios of state populations should somehow closely match the ratios of their numbers of electoral votes.  Apportionment methods were hotly debated over the nation's history; for more on this history see e.g. \cite{balinski75}.

As noted by Banzhaf in 1968, the probability of one single voter changing the election's outcome, if all other voters cast their votes randomly, tends to be higher for voters in larger states.  However, the assumption that all other voters cast their votes uniformly at random is unrealistic.  Despite our similarly unrealistic assumptions of voter behavior, i.e. Assumption \ref{voteas}, actual data for presidential elections in the U.S. (in Table \ref{table5}) demonstrates that it is much more likely for a small number of vote changes to change the electoral college's outcome than a plurality vote.

\begin{table}[htbp!]\label{table5}
\small
\centering
\begin{tabular}{crrlcc|}
  &  \multicolumn{5}{c}{U.S. Presidential Election Vote Margins}\\
 \cline{2-6}
  Elec-& \multicolumn{1}{|c}{Popular Vote } & Vote Changes  &Percent & Electoral  & State where votes would \\
   tion& \multicolumn{1}{|c}{Margin, } & Sufficient to  & of Popular &Vote & be changed (corresponding\\
   Year& \multicolumn{1}{|c}{Rounded} & Sway Election & Vote &Margin & electoral votes) \\
 \cline{1-6}
  \multicolumn{1}{|c|}{1844} & 40,000 &2,554 & .09\% &65 & New York (36)\\  
 \multicolumn{1}{|c|}{1848} & 140,000 & 6,669 & .23\% &36 & Pennsylvania (26)\\ 
\multicolumn{1}{|c|}{1856} & 500,000  & 11,155 &.28\% & 60 & Illinois, Tennessee, Kentucky (35) \\ 
 \multicolumn{1}{|c|}{\textcolor{red}{1876}} & \textcolor{red}{-250,000}  & 445 & .005\%  & 1 & South Carolina (7)\\   
 \multicolumn{1}{|c|}{\textcolor{blue}{1880}} & \textcolor{blue}{2,000}$^{*}$  & 8,416 & .09\% & 59 & OR, CT, CO, NH, IN (32)\\ 
 \multicolumn{1}{|c|}{1884} & 60,000 & 575  & .006\% & 37 & New York (36)\\ 
 \multicolumn{1}{|c|}{\textcolor{red}{1888}} & \textcolor{red}{-90,000}   & 7,187 & .06\% & 65 & NY (36)\\ 
 \multicolumn{1}{|c|}{1892} & 400,000   & 25,362 & .21\%  & 132 & CA, IN, ND, KA, WI, WV, IL (68)\\ 
 \multicolumn{1}{|c|}{1896} & 600,000   & 18,602 & .13\%  & 95 & KY, CA, OR, IN,  WV, DE (50)\\ 
 \multicolumn{1}{|c|}{1916} & 600,000   & 1,887  &  .01\% & 23 & California (13)\\ 
 \multicolumn{1}{|c|}{1948} & 2,200,000  & 29,294 & .06\% & 114 & OH, CA, IL (78)\\ 
 \multicolumn{1}{|c|}{1960} & 110,000   & 14,265  & .02\% & 84 & HI, IL, MO, SC (59)\\ 
 \multicolumn{1}{|c|}{1968} & 500,000   & 41,971 & .06\% & 110 & Missouri, New Jersey, Alaska (32)$^{**}$\\ 
 \multicolumn{1}{|c|}{1976} & 1,700,000  & 12,791 & .02\%  & 57 & Ohio and Mississippi (32)\\ 
 \multicolumn{1}{|c|}{\textcolor{red}{2000}} & \textcolor{red}{-500,000}   & 269 & .0003\% & 5 & Florida (25)\\ 
 \multicolumn{1}{|c|}{2004} & 3,000,000  & 59,301 & .05\% & 35 & Ohio (20)\\ 
 \multicolumn{1}{|c|}{2008} & 10,000,000  & 495,310 & .38\% & 192 & NC, IN, FL, OH, VA, IA, NH (97)\\ 
 \multicolumn{1}{|c|}{2012} & 5,000,000  & 214,764 & .17\% & 126 & FL, OH, VA, NH (64)\\ 
 \multicolumn{1}{|c|}{\textcolor{red}{2016}} & \textcolor{red}{-3,000,000}  & 38,875 & .03\%  & 77 & MI, PA, WI (46)\\ 
\cline{1-6}
\end{tabular}
\caption{\tiny In 17 of the country's 58 elections between 1788 and 2016, the popular vote was so narrow that changing a relatively small number of votes in just a few states would have shifted the result of the national election. In some years, \textcolor{red}{the person elected president lost the popular vote}. In one year, 1880, \textcolor{blue}{the Electoral College vote was just about as close as the popular vote}.  $^{*}$Historians disagree about the popular vote margin in the 1880 election.
$^{**}$In 1968, the House of Representatives was controlled by a different party than won the presidential election, so changing the election's outcome would have only required the winner to fail to receive a majority in the Electoral College.}
\vspace{-1cm}
\end{table}

%
%
%

\section{Other Applications}\label{otherapp}

As mentioned above, Majority is Stablest and Plurality is Stablest are closely related to geometric optimization problems involving soap bubbles.  For a general introduction to minimal surfaces, see the surveys \cite{colding19,colding11} or the book by the same authors.  For more discussion on the connections between voting and geometry, see the surveys \cite{odonnell14b,khot10b}.

 In 2002, it was proven that the two regions of fixed volume that minimize their total surface area are those pictured in Figure \ref{doublepic} \cite{hutchings02}.  The analogous result for three regions, as in Figure \ref{tripic}, is still open.  This problem is only solved in the plane by Wichiramala.  Surprisingly, the Gaussian versions of these problem were recently resolved in \cite{milman18b}, and then strengthened in \cite{heilman18b}.

 The initial motivation for the Majority is Stablest Theorem \ref{misg} and the Plurality is Stablest Conjecture \ref{pisinf} came from theoretical computer science.  These inequalities imply sharp computational hardness for MAX-CUT and its generalizations.  That is, we can efficiently, approximately solve some computational problem, and improving on this approximation is impossible to do efficiently, assuming the Unique Games Conjecture, a standard complexity theoretic assumption.  For more on the relation between voting and computer science applications, see \cite{khot10b,khot07,isaksson11}.

The noise stability of functions, as used in the Majority is Stablest Theorem \ref{misg}, has developed into a subject of its own.  Various references exist on the subject, such as \cite{diakon10}.

Besides the applications of voting mentioned above, voting is also used as a subroutine in various machine learning algorithms, such as ``boosting'' algorithms of Freund and Schapire.  In a ``boosting'' algorithm, one has access to several ``weak'' learning algorithms (or ``weak'' experts) who can each correctly classify e.g. an email as spam or not spam, with 51\% probability.  (The experts are called ``weak'' since it is easy to correctly classify an email as spam or not spam with 50\% probability, just by randomly choosing either spam or not spam, with equal probability.)  Using an appropriately chosen weighted majority vote among all of the classifications of these experts, their aggregate classification of the email can be correct with close to 100\% probability.  So-called ``boosting'' algorithms combine ``weak'' expert opinions to ``boost'' the probability of correct classification in this way.

\medskip
\noindent\textbf{Acknowledgement}.  Thanks to Daniel Kane and Elchanan Mossel for helpful discussions.

\bibliographystyle{amsalpha}
\bibliography{12162011}

\newcommand{\etalchar}[1]{$^{#1}$}
\def\polhk#1{\setbox0=\hbox{#1}{\ooalign{\hidewidth
  \lower1.5ex\hbox{`}\hidewidth\crcr\unhbox0}}} \def\cprime{$'$}
  \def\cprime{$'$}
\providecommand{\bysame}{\leavevmode\hbox to3em{\hrulefill}\thinspace}
\providecommand{\MR}{\relax\ifhmode\unskip\space\fi MR }
\providecommand{\MRhref}[2]{%
  \href{http://www.ams.org/mathscinet-getitem?mr=#1}{#2}
}
\providecommand{\href}[2]{#2}
\begin{thebibliography}{KKMO07}

\bibitem[ASS02]{arrow02}
Kenneth~J. Arrow, Amartya~K. Sen, and Kotaro Suzumura (eds.), \emph{Handbook of
  social choice and welfare. {V}ol. 1}, Handbooks in Economics, vol.~19,
  Elsevier/North-Holland, Amsterdam, 2002. \MR{3183780}

\bibitem[BCE{\etalchar{+}}16]{brandt16}
Felix Brandt, Vincent Conitzer, Ulle Endriss, J\'{e}r\^{o}me Lang, and Ariel~D.
  Procaccia (eds.), \emph{Handbook of computational social choice}, Cambridge
  University Press, New York, 2016. \MR{3587842}

\bibitem[BY75]{balinski75}
M.~L. Balinski and H.~P. Young, \emph{The quota method of apportionment}, Amer.
  Math. Monthly \textbf{82} (1975), no.~7, 701--730. \MR{504067}

\bibitem[CM11]{colding11}
Tobias~Holck Colding and William~P. Minicozzi, II, \emph{Minimal surfaces and
  mean curvature flow}, Surveys in geometric analysis and relativity, Adv.
  Lect. Math. (ALM), vol.~20, Int. Press, Somerville, MA, 2011, pp.~73--143.
  \MR{2906923}

\bibitem[CM19]{colding19}
\bysame, \emph{In search of stable geometric structures}, Notices Amer. Math.
  Soc. \textbf{66} (2019), no.~11, 1785--1791. \MR{3971084}

\bibitem[DHK{\etalchar{+}}10]{diakon10}
Ilias Diakonikolas, Prahladh Harsha, Adam Klivans, Raghu Meka, Prasad
  Raghavendra, Rocco~A. Servedio, and Li-Yang Tan, \emph{Bounding the average
  sensitivity and noise sensitivity of polynomial threshold functions},
  Proceedings of the Forty-second ACM Symposium on Theory of Computing (New
  York, NY, USA), STOC '10, ACM, 2010, pp.~533--542.

\bibitem[Hei19]{heilman18b}
Steven Heilman, \emph{Stable {G}aussian minimal bubbles}, Preprint,
  \href{https://arxiv.org/abs/1901.03934}{arXiv:1901.03934}, 2019.

\bibitem[HMN16]{heilman14}
Steven Heilman, Elchanan Mossel, and Joe Neeman, \emph{Standard simplices and
  pluralities are not the most noise stable}, Israel Journal of Mathematics
  \textbf{213} (2016), no.~1, 33--53.

\bibitem[HMRR02]{hutchings02}
Michael Hutchings, Frank Morgan, Manuel Ritor{\'e}, and Antonio Ros,
  \emph{Proof of the double bubble conjecture}, Ann. of Math. (2) \textbf{155}
  (2002), no.~2, 459--489. \MR{1906593 (2003c:53013)}

\bibitem[IM12]{isaksson11}
Marcus Isaksson and Elchanan Mossel, \emph{Maximally stable {G}aussian
  partitions with discrete applications}, Israel J. Math. \textbf{189} (2012),
  347--396. \MR{2931402}

\bibitem[Kal18]{kalai18}
Gil Kalai, \emph{Three puzzles on mathematics, computation, and games}, Notices
  of the American Mathematical Society \textbf{65} (2018).

\bibitem[Kho]{khot10b}
Subhash Khot, \emph{Inapproximability of {NP}-complete problems, discrete
  fourier analysis, and geometry}, Proceedings of the International Congress of
  Mathematicians 2010 (ICM 2010), pp.~2676--2697.

\bibitem[KKMO07]{khot07}
Subhash Khot, Guy Kindler, Elchanan Mossel, and Ryan O'Donnell, \emph{Optimal
  inapproximability results for {MAX}-{CUT} and other 2-variable {CSP}s?}, SIAM
  J. Comput. \textbf{37} (2007), no.~1, 319--357. \MR{2306295 (2008d:68035)}

\bibitem[KL20]{keevash18}
Peter Keevash and Eoin Long, \emph{Stability for vertex isoperimetry in the
  cube}, J. Combin. Theory Ser. B \textbf{145} (2020), 113--144. \MR{4102766}

\bibitem[MN18]{milman18b}
Emanuel Milman and Joe Neeman, \emph{The {G}aussian multi-bubble conjecture},
  Preprint, \href{https://arxiv.org/abs/1805.10961}{arXiv:1805.10961}, 2018.

\bibitem[MOO10]{mossel10}
Elchanan Mossel, Ryan O'Donnell, and Krzysztof Oleszkiewicz, \emph{Noise
  stability of functions with low influences: invariance and optimality}, Ann.
  of Math. (2) \textbf{171} (2010), no.~1, 295--341. \MR{2630040 (2012a:60091)}

\bibitem[MPR13]{mossel13}
Elchanan Mossel, Ariel~D. Procaccia, and Mikl\'{o}s~Z. R\'{a}cz, \emph{A smooth
  transition from powerlessness to absolute power}, J. Artif. Int. Res.
  \textbf{48} (2013), no.~1, 923–951.

\bibitem[O'D]{odonnell14b}
Ryan O'Donnell, \emph{Social choice, computational complexity, gaussian
  geometry, and boolean functions}, In proceedings of the 2014 ICM.

\bibitem[O'D08]{odonnell08}
Ryan O'Donnell, \emph{Some topics in analysis of boolean functions},
  Proceedings of the 40th Annual {ACM} Symposium on Theory of Computing,
  Victoria, British Columbia, Canada, May 17-20, 2008, 2008, pp.~569--578.

\bibitem[O'D14]{odonnell14}
Ryan O'Donnell, \emph{Analysis of {B}oolean functions}, Cambridge University
  Press, 2014.

\end{thebibliography}

\end{document}